\documentclass[a4paper, reqno,12pt]{amsart}
\usepackage{tikz, tikz-3dplot}
\usepackage{amsmath, amsthm, amssymb}

\theoremstyle{plain}
\newtheorem{te}{Theorem}[section]
\newtheorem{lem}[te]{Lemma}
\newtheorem{co}[te]{Corollary}

\newtheorem{de}[te]{Definition}
\newtheorem{ex}[te]{Example}

\newtheorem{qu}[te]{Question}
\newtheorem{con}[te]{Conjecture}
\theoremstyle{remark}
\newtheorem{re}[te]{Remark}
\newtheorem*{ack*}{Acknowledgment}

\def\x{{\textbf{x}}}

\def\n{{\bf n}}

\def\0{{\bf 0}}

\def\T{{\mathbb T}}
\def\R{{\mathbb R}}

\def\C{{\mathbb C}}
\def\S{{\mathbb S}}
\def\Z{{\mathbb Z}}
\def\P{{\mathbb P}}

\def\H{{\mathbb H}}

\def\supp{{\operatorname{supp}\,}}
\def\nint{\mathop{\diagup\kern-13.0pt\int}}

\def\dist{{\operatorname{dist}\,}}

\def\Ec{{\mathcal E}}

\def\Ac{{\mathcal A}}\def\Qc{{\mathcal Q}}
\def\Pc{{\mathcal P}}

\begin{document}

\title{Beyond canonical decoupling}

\author{Ciprian Demeter}
\address{Department of Mathematics, Indiana University, 831 East 3rd St., Bloomington IN 47405}
\email{demeterc@indiana.edu}

\thanks{The author is  partially supported by the NSF grant DMS-2055156}

\begin{abstract}
	We introduce two families of inequalities. Large ensemble decoupling is connected to the continuous restriction phenomenon. Tight decoupling is connected to the discrete Restriction conjecture for the sphere.
	Our investigation opens new grounds and answers some questions. 	
	
\end{abstract}

\maketitle

\section{The large ensemble problem}

Given $S\subset [0,1]^2$, $F:S\to\C$ and $x\in \R^3$, the restriction to the two dimensional  paraboloid $\P^2$ is
$$E_SF(x)=\int_SF(\xi)e((\xi,|\xi|^2)\cdot x)d\xi.$$
We write $EF$ for $E_{[0,1]^2}F$, and we denote by $B_R$ an arbitrary ball/cube of radius/side length $R$ in $\R^3$. For context, we recall the main result in \cite{BD}.

Let $\Qc_{R}$ denote the  $\frac1{\sqrt{R}}$-squares $q$ in $[0,1]^2$ of the form $[\frac{j}{\sqrt{R}},\frac{j+1}{\sqrt{R}}]\times [\frac{i}{\sqrt{R}},\frac{i+1}{\sqrt{R}}]$.
\begin{te}
\label{candec}		
	 For $2\le p\le 4$ we have  $$\|EF\|_{L^p(B_R)}\lesssim_{\epsilon}R^\epsilon (\sum_{q\in\Qc_R}\|E_{q}F\|_{L^p(B_R)}^2)^{1/2}.$$	
\end{te}
 To be fully precise, $L^p(B_R)$ on the right hand side should be $L^p(w_{B_R})$, where $w_{B_R}$ is a smooth approximation of $1_{B_R}$. To keep notation simple, this will be ignored, as it has no impact on our arguments.
This inequality is known as (canonical) $l^2$ decoupling. It immediately implies the superficially weaker inequality called (canonical) $l^p$ decoupling
$$\|EF\|_{L^p(B_R)}\lesssim_{\epsilon}R^{\frac12-\frac1p+\epsilon} (\sum_{q\in\Qc_R}\|E_{q}F\|_{L^p(B_R)}^p)^{1/p}.$$
There are similar decouplings for other manifolds, but we only focus on $\P^2$ here. The phenomenon we are about to describe is already complex in this case.

The scale $1/\sqrt{R}$ of the squares $q$ has a special relation to the scale $R$ of the balls $B_R$. It allows the crucial use of parabolic rescaling in the proof of Theorem \ref{candec}. For this reason, the squares $q$ (or their lifts to $\P^2$) are sometimes called {\em canonical} caps. One notable variant that has been considered is called small cap decoupling. This time $[0,1]^2$ is partitioned into smaller $\frac1{R^\alpha}$-squares, $\frac12< \alpha\le 1$, called {\em small caps}. See \cite{BDG} and \cite{GMO}.

Here, we go in the opposite direction and replace the squares $q$ with larger sets. Considering just larger squares does not lead to anything new or interesting. Instead, we will allow arbitrary unions of such squares. Examples include thin rectangles and disconnected sets.

\begin{de}
A disjoint union of $\frac1{\sqrt{R}}$-squares $q\in\Qc_R$  will be called a large ensemble.
\end{de}We kick-start our investigation with the following conjecture.

\begin{con}[Large ensemble decoupling for the paraboloid]
\label{c1}	
Let $N=O(R)$.	
Let $S_1,\ldots,S_N$ be a partition of $[0,1]^2$ into large ensembles. Then for $2\le p\le 3$ we have
\begin{equation}
\label{eltwolargeen}
\|EF\|_{L^p(B_R)}\lesssim_{\epsilon}R^\epsilon (\sum_{i=1}^N\|E_{S_i}F\|_{L^p(B_R)}^2)^{1/2}.
\end{equation}	
\end{con}
The case $p=2$ follows from $L^2$ orthogonality. As we will see in the next section, the case $p=3$ implies the full Restriction conjecture for $\P^2$. The same computations will show why the restriction $p\le 3$ is needed. No interpolation is available in this general setting.

A well known consequence of Theorem \ref{candec} is the fact that for each $2\le p\le 4$ and each exponential sum $\Ec(x)=\sum_{q}a_qe(\xi_q\cdot x)$ associated with $1/\sqrt{R}$-separated points $\xi_q$ on $\P^2$ we have
$$\|\Ec\|_{L^p(B_R)}\approx \|a_q\|_{l^2}|B_R|^{1/p}.$$	
Here and in the following we use the notation $\approx$ to denote an equality modulo $R^\epsilon$ multiplicative losses, for arbitrarily small $\epsilon>0$. This immediately shows that when $EF$ is such an exponential sum,  \eqref{eltwolargeen} holds with $\lesssim$ replaced with $\approx$, in the larger range $2\le p \le 4$, and for each $S_i$.
\smallskip

The reason we anticipate the validity of the conjecture for arbitrary $F$ is because it is implied by the following reverse square function estimate, widely believed to be true for $2\le p\le 3$
\begin{equation}
\label{eeeeeeee1}
\|EF\|_{L^p(B_R)}\lesssim_\epsilon\|(\sum_{q\in\Qc_R}|E_qF|^2)^{1/2})\|_{L^p(B_R)}.
\end{equation}

\begin{te}
	\label{condit}
Assume  \eqref{eeeeeeee1} is true. Then Conjecture \ref{c1} holds.
\end{te}
\begin{proof}
We first combine \eqref{eeeeeeee1} with Minkowski's inequality to get
$$\|EF\|_{L^p(B_R)}\lesssim (\sum_{i=1}^N\|(\sum_{q\in S_i }|E_{q}F|^2)^{1/2}\|^2_{L^p(B_R)})^{1/2}.$$
Note that $$\|(\sum_{q\in S_i }|E_{q}F|^2)^{1/2}\|_{L^p(B_R)}\approx \|(\sum_{q\in S_i }|\eta_RE_{q}F|^2)^{1/2}\|_{L^p(\R^3)},$$
where $\eta_R$ is a smooth approximation of $1_{B_R}$, with compactly supported Fourier transform. The Fourier transforms of the functions $\eta_RE_qF$ are supported on a finitely overlapping family of $1/\sqrt{R}$-cubes near $\P^2$.

To end, we use Rubio de Francia's square function estimate for this collection of cubes
$$\|(\sum_{q\in S_i }|\eta_RE_{q}F|^2)^{1/2}\|_{L^p(\R^3)}\lesssim \|\sum_{q\in S_i }\eta_RE_{q}F\|_{L^p(\R^3)}\approx \|E_{S_i}F\|_{L^p(B_R)}.$$
\end{proof}

In Section \ref{sec2} we prove that Conjecture \ref{c1} implies the Restriction conjecture for $\P^2$. Since proving it for arbitrary partitions is out of reach,  in Section \ref{sec3} we introduce methods aimed to solve a particular case of interest.

\begin{ack*}
	I am grateful to Larry Guth for providing me with the probabilistic heuristics for Corollary \ref{jkgimgiigf[i[pogi]]}. I thank Nets Katz for suggesting the name for large ensemble decoupling. The author has enjoyed conversations on this topic with Hong Wang and Aric Wheeler. 	
\end{ack*}
\section{Conjecture \ref{c1} implies the Restriction Conjecture}
\label{sec2}

Here is the main result in this section. A union of $M$ canonical squares $q$ is called an $M$-set.
\begin{te}
\label{trelation}	
Fix  $0<\gamma\le 1$.	
Assume that the following large ensemble $l^p$ decoupling holds for some $p\in [2,4]$, $M\sim R^{1-\gamma}$ (and all $\epsilon>0$, all $R\gtrsim 1$)
\begin{equation}
\label{sdcyre8we t}
\|EF\|_{L^p(B_R)}^p\lesssim_\epsilon R^\epsilon (\frac{R}{M})^{\frac{p}{2}-1}\sum_{j=1}^{R/M}\|E_{P_j}F\|_{L^p(B_R)}^p,\end{equation}
for all $F$ and all partitions $\Pc$ of $\Qc_R$ into $M$-sets $P_j$.
Then for each $\|F\|_\infty\lesssim R$ we have
$$\int_{B_{\sqrt{R}}}|EF|^p\lesssim_\epsilon R^\epsilon(R^{\frac{3+p}{2}}+R^3(\frac{R}{M})^{\frac{p}{2}-1}).$$
\end{te}
Before we prove the theorem, we reveal its significance with two immediate consequences.
\begin{co}
\label{jkgimgiigf[i[pogi]]}	
Assume \eqref{sdcyre8we t} holds true for $p=3$ and all $\gamma$ arbitrarily close to $0$. Then the Restriction conjecture for the paraboloid holds: for each $F$ we have
$$\int_{B_{\sqrt{R}}}|EF|^3\lesssim_\epsilon R^\epsilon\|F\|_\infty^3.$$	
\end{co}
\begin{co}
\label{cosecon}
Assume \eqref{sdcyre8we t} holds true for some $p\in(3,4]$ and  $M\sim R^{\frac{4-p}{p-2}}$. Then the Restriction conjecture for the paraboloid holds in $L^p$: for each $F$ we have
$$\int_{B_{\sqrt{R}}}|EF|^p\lesssim_\epsilon R^\epsilon\|F\|_\infty^p.$$
Moreover, if $\Delta>0$ then \eqref{sdcyre8we t} is false for  $p\in(3,4]$, $M\sim R^{\frac{4-p}{p-2}+\Delta}$ and some partition into $M$-sets.	
\end{co}
To see the last part of Corollary \ref{cosecon}, note that if \eqref{sdcyre8we t} held for some $M\sim R^{\frac{4-p}{p-2}+\Delta}$,  then Theorem \ref{trelation} would imply the false inequality (use $F=1_{[0,1]^2}$)
$$\int_{B_{\sqrt{R}}}|EF|^p\lesssim R^{-\beta}\|F\|_\infty^p,$$
for some $\beta>0$.
	
The next two lemmas establish average $l^2$ and $l^4$ estimates, via elementary expansion and counting arguments. These will later be interpolated to get $l^3$ estimates. Such pointwise estimates are then integrated on $B_R$ to get $L^3$ estimates. The idea is that averaging over partitions produces square root cancellation. When combined with the conjectured decoupling into large ensembles, this leads to restriction estimates as described above.

In the following, sums are taken over all partitions $\Pc$ of $\{1,\ldots,R\}$ into sets $P_j$ with uniform size $M\le R$. Partitions are always considered unordered.

\begin{lem}For each $a_i\in\C$, $1\le i\le R$
$$\frac1{\#\Pc}\sum_{\Pc}\sum_{P_j\in\Pc}|\sum_{i\in P_j}a_i|^2\lesssim \frac{M}{R}|\sum_ia_i|^2+\sum_{i}|a_i|^2.$$
\end{lem}
\begin{proof}We have
$$\sum_{\Pc}\sum_{P_j\in\Pc}|\sum_{i\in P_j}a_i|^2=\sum_{\Pc}\sum_{i}|a_i|^2+\sum_{i_1\not=i_2}\#\Pc(i_1,i_2)\overline{a_{i_1}}a_{i_2}$$
where $\Pc(i_1,i_2)$ are the special partitions that put ${i_1},{i_2}$ into the same cell. Note that we have
$$\frac{1}{(R/M)!}{R\choose M} {R-M\choose M}\ldots{2M\choose M}$$
partitions $\Pc$. Indeed, we first select the first cell, which can be done in $R\choose M$ ways. The second to last cell can be selected in $2M\choose M$ ways. We need to divide with $(R/M)!$, as the partition is not ordered.

The number of partitions $\Pc(i_1,i_2)$ is
$$\frac{1}{(\frac{R}{M}-1)!}{R-2\choose M} {R-M-2\choose M}\ldots{2M-2\choose M}.$$
We select the first $\frac{R}{M}-1$ cells from $[1,R]\setminus\{i_1,i_2\}$, and put $i_1,i_2$ in the small cell  left in the end. The number of selected cells needs to be divided by the number of permutations. Note that
$$\#\Pc(i_1,i_2)=\#\Pc\times\frac{R}{M}\frac{M(M-1)}{R(R-1)}\sim \#\Pc M/R.$$
We conclude that
$$\sum_{\Pc}\sum_{P_j\in\Pc}|\sum_{i\in P_j}a_i|^2=$$$$\#\Pc(\sum_{i}|a_i|^2+\frac{M-1}{R-1}\sum_{i_1\not=i_2}\overline{a_{i_1}}a_{i_2})=\#\Pc(\frac{R-M}{R-1}\sum_{i}|a_i|^2+\frac{M-1}{R-1}|\sum_ia_i|^2).$$

\end{proof}

\begin{lem}For each $a_i\in\C$
$$\frac1{\#\Pc}\sum_{\Pc}\sum_{\;P_j\in\Pc}|\sum_{i\in P_j}a_i|^4\lesssim $$$$\frac{M^3}{R^3}|\sum_ia_i|^4+\frac{M^2}{R^2}\sum_i|a_i|^2|\sum_ia_i|^2+\frac{M}{R}\sum_i|a_i|^{3}\sum_i|a_i|+\sum_{i}|a_i|^4.$$
\end{lem}
\begin{proof}
The expansion of the term on the left will contain four types of quadruple products.

First, there are the terms $a_{i_1}\overline{a_{i_2}}a_{i_3}\overline{a_{i_4}}$ with $i_1\not=i_2\not=i_3\not=i_4$. Reasoning as in the previous lemma, each such term appears roughly $M^3\#\Pc/R^3$ many times in the summation, as this is the number of partitions that place the four indices in the same cell.

Second, there are the terms of the form $a_{i_1}\overline{a_{i_1}}a_{i_2}\overline{a_{i_3}}$ (and their permutations) with $i_1\not=i_2\not=i_3$. These collect all products where exactly three of the four indices are distinct. Each such term appears roughly $M^2\#\Pc/R^2$ many times.

Third, we will have products containing two distinct indices, of the form $a_{i_1}\overline{a_{i_1}}a_{i_2}\overline{a_{i_2}}$, $a_{i_1}\overline{a_{i_1}}a_{i_1}\overline{a_{i_2}}$ and their permutations. Each of those appear roughly $M\#\Pc/R$ many times.

Finally, each term $|a_i|^4$ appears exactly $\#\Pc$ many times.

The contributions from the first category is dominated by (we add the last term to ``complete the square")
$$\lesssim\frac{M^3\#\Pc}{R^3}(|\sum_ia_i|^4+|\sum_{i_1,i_2,i_3,i_4\atop{\text{not all four distinct}}}a_{i_1}\overline{a_{i_2}}a_{i_3}\overline{a_{i_4}}|).$$
We also have the following estimate for the term we added
$$\frac{M^3\#\Pc}{R^3}|\sum_{i_1,i_2,i_3,i_4\atop{\text{not all four distinct}}}a_{i_1}\overline{a_{i_2}}a_{i_3}\overline{a_{i_4}}|\le$$$$ \frac{M^3\#\Pc}{R^3}(|\sum_{i_1,i_2,i_3,i_4\atop{\text{three distinct indices}}}a_{i_1}\overline{a_{i_2}}a_{i_3}\overline{a_{i_4}}|+|\sum_{i_1,i_2,i_3,i_4\atop{\text{two distinct indices}}}a_{i_1}\overline{a_{i_2}}a_{i_3}\overline{a_{i_4}}|+\sum_{i}|a_i|^4).$$
These contributions are absorbed by the remaining three terms, as the exponent three in the coefficient $(M/R)^3$ here is larger than those in the remaining cases.

Indeed, the contribution from the second case is
$$\lesssim \frac{M^2\#\Pc}{R^2}|\sum_{i_1,i_2,i_3,i_4\atop{\text{three distinct indices}}}a_{i_1}\overline{a_{i_2}}a_{i_3}\overline{a_{i_4}}|.$$	
Let us analyze the case $i_1=i_2$, the other ones are similar. Then, completing the square again
$$\frac{M^2\#\Pc}{R^2}|\sum_{i_1,i_3,i_4\atop{{i_1\not=i_3\not=i_4}}}|a_{i_1}|^2a_{i_3}\overline{a_{i_4}}|\le \frac{M^2\#\Pc}{R^2}(\sum_{i}|a_i|^2|\sum_{i}a_i|^2+\sum_{i_1,i_2,i_3\atop{\text{not all three distinct}}}|a_{i_1}|^2a_{i_3}\overline{a_{i_4}}).$$	
As before, the term we added is absorbed by the contributions of the remaining two cases.

Indeed, the contribution from the third case is
$$\frac{M\#\Pc}{R}|\sum_{i_1,i_2,i_3,i_4\atop{\text{only two distinct indices}}}a_{i_1}\overline{a_{i_2}}a_{i_3}\overline{a_{i_4}}|.$$
This is dominated by$$
\frac{M\#\Pc}{R}((\sum_{i}|a_i|^2)^2+\sum_{i_1}|a_{i_1}|^3\sum_{i_2}|a_{i_2}|).$$
The second term dominates because of Cauchy--Schwarz.
The contribution from the fourth case is
$$\#\Pc\sum_{i}|a_i|^4.$$

\end{proof}
Next, we interpolate.
\begin{co}
	\label{elp}
Assume $|a_i|\lesssim 1$ and let $L=\sum_{i=1}^R|a_i|$,  $S=|\sum_{i=1}^Ra_i|$.
Then for  $2\le p\le 4$
$$\frac1{\#\Pc}\sum_{\Pc}\sum_{\;P_j\in\Pc}|\sum_{i\in P_j}a_i|^p\lesssim (\frac{M}{R})^{\frac{p}{2}-1}L^{p/2}+L+(\frac{M}{R})^{p/2}S^p.$$
\end{co}
\begin{proof}	
Recall that
$$\frac1{\#\Pc}\sum_{\Pc}\sum_{P_j\in\Pc}|\sum_{i\in P_j}a_i|^2\lesssim \frac{M}{R}S^2+L$$
and
$$\frac1{\#\Pc}\sum_{\Pc}\sum_{\;P_j\in\Pc}|\sum_{i\in P_j}a_i|^4\lesssim \frac{M^3}{R^3}S^4+\frac{M^2}{R^2}LS^2+\frac{M}{R}L^2+L.$$
We combine these with the inequality $\|\cdot\|_{l^p}^p\le \|\cdot\|_{l^2}^{4-p}\|\cdot\|_{l^4}^{2p-4}$.
\\
\\
Case 1. If $L\ge \frac{M}{R}S^2$ then we get
$$\frac1{\#\Pc}\sum_{\Pc}\sum_{\;P_j\in\Pc}|\sum_{i\in P_j}a_i|^p\lesssim L^{2-\frac{p}2}(\frac{M}{R}L^2+L)^{\frac{p}2-1}\lesssim (\frac{M}{R})^{\frac{p}{2}-1}L^{p/2}+L.$$
\\
\\
Case 2. If  $L\le \frac{M}{R}S^2$ then the term $\|\cdot \|_{l^4}^4$ is dominated by $\frac{M^3}{R^3}S^4+L$. Since $S\le L$ and $M\le R$, we must have that $L\ge 1$. Thus $L\le L^2\le \frac{M^2}{R^2}S^4$, and we may further dominate  $\|\cdot \|_{l^4}^4$ by $\frac{M^2}{R^2}S^4$.
We get
$$\frac1{\#\Pc}\sum_{\Pc}\sum_{\;P_j\in\Pc}|\sum_{i\in P_j}a_i|^p\lesssim (\frac{M}{R}S^2)^{2-\frac{p}2}(\frac{M^2}{R^2}S^4)^{\frac{p}{2}-1}\lesssim (\frac{M}{R})^{p/2}S^p.$$

\end{proof}
\medskip

We now verify Theorem \ref{trelation}.
\begin{proof}
The first part of the argument follows the lines of \cite{Ca}.
We may assume $B_{\sqrt{R}}$ is centered at the origin. Due to the smoothing effect of the convolution with the Fourier transform of the characteristic function of this ball, we may assume $F$ is constant on each $q$. For such $F$, we will in fact prove the superficially stronger inequality on the larger ball $B_R$ centered at the origin
$$\int_{B_{R}}|EF|^p\lesssim_\epsilon R^\epsilon(R^{\frac{3+p}{2}}+R^3(\frac{R}{M})^{\frac{p}{2}-1}).$$

On $B_R$, each $E_{q_i}F$ is a single wave packet $\Phi_i$ that is spatially localized near a $(R,R^{1/2},R^{1/2})$-tube $T_i$ through the origin, and has $L^\infty$ norm $\lesssim 1$. We may write this as
$$E_{q_i}F(x)=\Phi_i(x)1_{R^\delta T_i}(x)+O_\delta(R^{-1000}),\;\;x\in B_R,$$
for $\delta$ as small as we wish.  The first key inequality we use is the well-known bush estimate
$$\|\sum_{i=1}^R1_{R^\delta T_i}\|_{r}\lesssim R^{1+\frac3{2r}+O(\delta)}$$
valid for $r\ge 3/2$. This follows from an easy computation. It shows that the main contribution to the  integral comes from the core $B_{\sqrt{R}}$ where all tubes intersect. This justifies the fact that moving integration from this smaller ball to $B_R$ is not lossy.

We average \eqref{sdcyre8we t} over all partitions  $\Pc$ into $M$-sets to get
\begin{equation}
\label{eufyrfucfr}
\|EF\|_{L^p(B_R)}^p\le C_R  (\frac{R}{M})^{\frac{p}2-1}\frac{1}{\#\Pc}\sum_{\Pc}\sum_{P_j\in\Pc}^{R/M}\|\sum_{q_i\in P_j}E_{q_i}F\|_{L^p(B_R)}^p,\end{equation}
with $C_R\lesssim_\epsilon R^\epsilon$. Writing
$a_i(x)=\Phi_i(x)1_{R^\delta T_i}(x)$, we have
$$(\frac{R}{M})^{\frac{p}{2}-1}\frac{1}{\#\Pc}\sum_{\Pc}\sum_{P_j\in\Pc}^{R/M}|\sum_{q_i\in P_j}E_{q_i}F(x)|^p=(\frac{R}{M})^{\frac{p}{2}-1}\frac{1}{\#\Pc}\sum_{\Pc}\sum_{P_j\in\Pc}^{R/M}|\sum_{q_i\in P_j}a_i(x)|^p+O_\delta(R^{-100}).$$
Write
\begin{equation}
\label{eror}
S(x)=|\sum_ia_i(x)|=|EF(x)|+O_\delta(R^{-100})
\end{equation}
and
$$L(x)=\#\{i:\;x\in 1_{R^\delta T_i}\}.$$
Corollary \ref{elp} implies that
$$(\frac{R}{M})^{\frac{p}{2}-1}\frac{1}{\#\Pc}\sum_{\Pc}\sum_{P_j\in\Pc}^{R/M}|\sum_{q_i\in P_j}a_i(x)|^p\lesssim L(x)^{p/2}+(\frac{R}{M})^{\frac{p}{2}-1}L(x)+\frac{M}{R}S(x)^3.$$
We integrate this over $B_R$ and combine with \eqref{eufyrfucfr} and \eqref{eror} to conclude that
$$\|EF\|_{L^p(B_R)}^p\lesssim C_R(\int_{B_R}L^{p/2}+(\frac{R}{M})^{\frac{p}2-1}\int L+\frac{M}{R}\|EF\|_{L^3(B_R)}^3)+O_\delta(R^{-10}).$$
If $M\ll R/C_R$ this leads to
$$\|EF\|_{L^p(B_R)}^p\lesssim C_R(\int_{B_R}L^{p/2}+(\frac{R}{M})^{\frac{p}2-1}\int_{B_R} L)+O_\delta(R^{-10}).$$
Using the bush estimate and the trivial
 $\int_{B_R}L\lesssim R^{3+O(\delta)}$, we get
$$\|EF\|_{L^p(B_R)}^p\lesssim_\epsilon R^\epsilon(R^{\frac{3+p}{2}}+R^3(\frac{R}{M})^{\frac{p}{2}-1}),$$
by choosing $\delta$ small.
\end{proof}
The conclusion of Corollary \ref{jkgimgiigf[i[pogi]]}	 remains true if we average over a more restricted collection of  partitions of the following form. We fix $M$ collections $S_k=S_k(M)$ each consisting of $R/M$  squares $q$, in such a way that the $S_k$ form a partition of all squares. The partitions $\Pc$ we will use in the next theorem consist of $M$-sets $P_j$, each having exactly one element from each $S_k$. There are $(\frac{R}{M}!)^{M-1}$ such partitions. The proof is left as an exercise to the reader.
\begin{te}
\label{sparse}	
Assume that for all $M$ sufficiently close to $R$ we have	
$$\|EF\|_{L^3(B_R)}^3\lesssim_\epsilon R^\epsilon \frac1{\#\Pc}\sum_{\Pc=\{P_j\}}(\frac{R}{M})^{1/2}\sum_{j=1}^{R/M}\|E_{P_j}F\|_{L^3(B_R)}^3.$$
Then the Restriction conjecture for $\P^2$ holds.
\end{te}

The situation is even more dramatic for $l^2$ decoupling.
\begin{te}
\label{warning}	
Fix $\beta\in (0,1)$. Assume \eqref{eltwolargeen} holds for all partitions into $R^\beta$-sets $S_i$ when $p=3$. Then the Restriction conjecture for $\P^2$ holds.
\end{te}

\begin{proof}Fix $\delta$ such that $\frac{\beta-1}{2}+\delta\le 0$. Write $R'=R^{\beta+2\delta}$ and $M=R^{\beta}=(R')^{\frac{\beta}{\beta+2\delta}}$. We consider $F$ supported  on $[0,R^{\frac{\beta-1}{2}+\delta}]^2$. Inequality \eqref{eltwolargeen} and H\"older's inequality implies that for each such $F$ and each partition $\Pc=\{P_j\}$ of the collection of $1/\sqrt{R}$-squares $q$ in $[0,R^{\frac{\beta-1}{2}+\delta}]^2$ into $M$-sets $P_j$ we have
$$\|EF\|_{L^3(B_R)}^3\lesssim_\epsilon R^\epsilon (\frac{R'}{M})^{1/2}\sum_{j=1}^{R'/M}\|E_{P_j}F\|_{L^3(B_{R})}^3.$$	
Using parabolic rescaling $\xi=(\xi_1,\xi_2)\to \xi'=R^{\frac{1-\beta}{2}-\delta}(\xi_1,\xi_2)$ on the ball $B_R=[-R,R]^3$, this can be written as 	
$$\|E\tilde{F}\|_{L^3([-R^{\frac{\beta+1}{2}+\delta},R^{\frac{\beta+1}{2}+\delta} ]^2\times[-R',R'])}^3\lesssim_\epsilon (R')^{\epsilon} (\frac{R'}{M})^{1/2}\sum_{j=1}^{R'/M}\|E_{P_j}\tilde{F}\|_{L^3([-R^{\frac{\beta+1}{2}+\delta},R^{\frac{\beta+1}{2}+\delta} ]^2\times[-R',R'])}^3.$$
This holds true for all $\tilde{F}:[0,1]^2\to\C$ and all  partitions of $[0,1]^2$ into $M$-sets $P_j$ consisting of  $1/\sqrt{R'}$-squares $q'$.
We will use the superficially weaker consequence
 \begin{equation}
 \label{eiofufwll;slplfp}
 \|E\tilde{F}\|_{L^3([-R',R']^3)}^3\lesssim_\epsilon (R')^\epsilon (\frac{R'}{M})^{1/2}\sum_{j=1}^{R'/M}\|E_{P_j}\tilde{F}\|_{L^3(\R^2\times[-R',R'])}^3.\end{equation}
 The rest of the argument is identical to the one for Theorem \ref{trelation}. We consider an arbitrary $\tilde{F}$ with $\|\tilde{F}\|_{\infty}\lesssim R'$ and prove that
 $$\int_{[0,\sqrt{R'}]^3}|\tilde{F}|^3\lesssim_\epsilon (R')^{3+\epsilon}.$$
 We may assume $\tilde{F}$ is constant on each $q'$. For such $\tilde{F}$ we prove the superficially stronger inequality
  \begin{equation}
  \label{mjerufue89u8uv8}
  \|E\tilde{F}\|_{L^3([-R',R']^3)}^3\lesssim_\epsilon (R')^{3+\epsilon}.
  \end{equation}
  We then average \eqref{eiofufwll;slplfp} over all partitions, and proceed as before to find
  $$\int_{[-R',R']^3}|E\tilde{F}|^p\lesssim_\epsilon (R')^\epsilon((R')^{3}+(R')^3(\frac{R'}{M})^{3/2}).$$
  We have used that $$\|E_{P_j}\tilde{F}\|_{L^3(\R^2\times[-R',R'])}^3\approx \|E_{P_j}\tilde{F}\|_{L^3([-R',R']^3)}^3.$$
  This is because the mass of $E_{P_j}\tilde{F}$ is concentrated in a tube that lives in $[-R',R']^3$.

  Recall that $R'/M=R^{2\delta}$. The estimate above holds uniformly for each $R'\gg 1$ and $\delta$. Inequality \eqref{mjerufue89u8uv8} is seen to hold by letting $\delta,\epsilon\to0$.

\end{proof}

When $M$ is small, the validity of \eqref{sdcyre8we t} seems much weaker than the full Restriction conjecture. For example, the case $M=1$ is Theorem \ref{candec}. The value $M=R^{1/2}$ is particularly appealing, as \eqref{sdcyre8we t} for this $M$ and $p=10/3$ implies (via Corollary \ref{cosecon}) the sharp  restriction estimate for $p=10/3$. This exponent, initially established by Tao \cite{Ta}, was a landmark threshold until it was reproved and slightly  improved in \cite{BG}. This suggests that perhaps \eqref{sdcyre8we t} could be proved with existing technology when $M$ is $R^{1/2}$ (or smaller). We caution that the $l^2$ version of this when $p=3$ implies the Restriction conjecture, see Theorem \ref{warning}.

\begin{con}
	\label{c2}
	Assume each $S_i$ is a $\sqrt{R}$-set.
	The inequality
	\begin{equation}
	\label{cjfhur}
	\|EF\|_{L^p(B_R)}\lesssim_{\epsilon}R^{\frac12(\frac12-\frac1p)+\epsilon} (\sum_{i=1}^{R^{1/2}}\|E_{S_i}F\|_{L^p(B_R)}^p)^{1/p}
	\end{equation}
	holds for $2\le p\le 10/3$.	
\end{con}

We offer evidence for the conjecture by verifying the case
when $EF=\sum_T EF_T$, with $EF_T$ a single wave packet corresponding to each $q$ (so $T=T_q$), $|EF_T|\approx 1_T$ on $B_R$. We start with the endpoint $p=10/3$.

First, an easy consequence of Tao's bilinear theorem \cite{Ta} is the  reverse square function estimate (see e.g. Exercise 7.25 in \cite{Dbook})
$$\|EF\|_{L^{10/3}(B_R)}\lesssim_\epsilon R^{1/20+\epsilon}\|(\sum_T |EF_T|^2)^{1/2}\|_{L^{10/3}(B_R)}.$$
This inequality holds true for arbitrary $F$, but is sharp in our case, when all tubes intersect.

Next, note that
$$\|(\sum_T |EF_T|^2)^{1/2}\|_{L^{10/3}(B_R)}\approx \|\sum_T 1_T\|_{L^{5/3}}^{1/2}.$$

Then, using Wolff's Kakeya estimate \cite{Wo} (there is one tube per direction) we have
$$\|\sum_T 1_T\|_{5/3}\lesssim_\epsilon \|\sum_T 1_T\|_{L^1}^{3/5}R^{1/10+\epsilon}.$$
Putting these together we get
$$
\|EF\|_{L^{10/3}(B_R)}\lesssim_\epsilon  \|\sum_T 1_T\|_{L^1}^{3/10}R^{1/10+\epsilon}.
$$
Finally, note that by using Rubio de Francia's square function estimate as before we find
$$\|E_{S_i}F\|_{L^{10/3}(B_R)}\gtrsim \|(\sum_{T=T_q\atop{q\subset S_i}}|EF_T|^2)^{1/2}\|_{L^{10/3}(B_R)}\approx  \|\sum_{T=T_q\atop{q\subset S_i}} 1_T\|_{L^{5/3}}^{1/2}\gtrsim \|\sum_{T=T_q\atop{q\subset S_i}} 1_T\|_{L^{1}}^{3/10}.$$
The last inequality is due to the fact that the sum is either zero or greater than 1.

When $p<10/3$, the inequality loses strength. The upper bound  $\|EF\|_{L^p(B_R)}\lesssim R^{\frac{5}{2p}+\frac14}$  follows by interpolating  the $L^{10/3}$ bound $O(R)$ from above  with the trivial $L^2$ bound $O(R^{3/2})$. This estimate is not sharp, except for the endpoints, but suffices for our purpose. More precisely, the Restriction conjecture
 anticipates that $\|EF\|_{L^p(B_R)}\lesssim_\epsilon R^{1+\epsilon}$ for $p\ge 3$.

 Lower bounds for $\|E_{S_i}F\|_{L^{p}(B_R)}$ follow as before, and produce the lower bound  $R^{\frac{5}{2p}+\frac14}$ for the right hand side of \eqref{cjfhur}.

\medskip

In the next section, we prove an unconditional result for $\sqrt{R}$-sets that are one dimensional in nature.

\section{An unconditional large ensemble decoupling}
\label{sec3}

Here is the main result in this section.

\begin{te}
	\label{t2}
Let $\sqrt{R}\lesssim N\le R$.	
Let $S_1,\ldots,S_N$ be a partition of $[0,1]^2$ into large ensembles, such that each square $Q\subset [0,1]^2$ with side length $L\ge R^{-1/2}$
intersects at most $O(LN)$ of these sets. Then for each $F:[0,1]^2\to\C$ and $2\le p\le 3$ we have
$$\|EF\|_{L^p(B_R)}\lesssim_{\epsilon}R^\epsilon N^{\frac12-\frac1p}(\sum_{i=1}^N\|E_{S_i}F\|_{L^p(B_R)}^p)^{1/p}.$$

\end{te}

The proof relies on recoupling. We formulate it in a way that fits our needs here.
\begin{te}[Recoupling]
Assume  $F$ is supported on pairwise  disjoint squares $\alpha_1, \ldots,\alpha_M$ with  side length $1/K$. Then for each ball $B$ with radius at least $K$ and for each $p\ge 2$ we have
$$ (\sum_{i=1}^M\|E_{\alpha_i}F\|_{L^p(B)}^p)^{1/p}\lesssim \|EF\|_{L^p(B)}.$$
\end{te}
\begin{proof}
	The result is immediate when $p\in\{2,\infty\}$.
Use  $l^2(L^2)-l^\infty(L^\infty)$ interpolation. Alternatively, use Rubio de Francia's square function estimate, as in the proof of Theorem \ref{condit}.

\end{proof}
We also use the following unconditional version of Theorem \ref{condit}. Its proof mirrors the one of Theorem \ref{condit}, using the trilinear reverse square function estimate in \cite{BCT}.
\begin{te}
\label{t1}	
 The trilinear version of Conjecture \ref{c1} holds. More precisely, if $F=F_1+F_2+F_3$ and $F_i$ are supported transversely, then for $2\le p\le 3$
$$\|(\prod_iEF_i)^{1/3}\|_{L^p(B_R)}\lesssim_{\epsilon}R^\epsilon (\sum_{i=1}^N\|E_{S_i}F\|_{L^p(B_R)}^2)^{1/2}.$$
\end{te}

\begin{proof}(of Theorem \ref{t2})
	
	Let $C_R$ be the best constant such that
$$\|EF\|_{L^p(B_R)}\le C_R N^{\frac12-\frac1p}(\sum_{i=1}^N\|E_{S_i}F\|_{L^p(B_R)}^p)^{1/p}$$
holds for all $F$, $B_R$,  $N$ and  $S_i$ as in the statement of Theorem \ref{t2}. Fix $K=O(1)$, large enough.
\medskip

We follow a  standard Bourgain--Guth argument, see \cite{BG}. Split $[0,1]^2$ into $1/K$-squares $\alpha$. Consider a finitely overlapping cover of $B_R$ with balls $B_K$. For each $B_K\subset B_R$, there is a flat $1\times 1/K$-strip $S$ (that is, the collection of all $1/K$-squares $\alpha$ whose slight enlargements intersect a given line) inside $[0,1]^2$, and there are transverse (that is, any three points from the three squares determine an area $\gtrsim_K 1$) $1/K$-squares $\alpha_1,\alpha_2,\alpha_3$ such that
$$|EF(x)|\lesssim K^{O(1)}|\prod_{i=1}^3 E_{\alpha_i}F(x)|^{1/3}+|\sum_{\alpha\subset S}E_\alpha F(x)|,\;x\in B_K.$$
Small cap $l^p(L^p)$ decoupling for the parabola (\cite{DGW}) gives
$$\int_{B_K}|\sum_{\alpha\subset S}E_\alpha F|^p\lesssim K^{\frac{p}2-1}\sum_{\alpha\subset S}\int_{B_K}|E_{\alpha}F|^p.$$
Summing up over $B_K$ we get
\begin{equation}
\label{e1}
\|EF\|_{L^p(B_R)}^p\lesssim_{\epsilon}K^{O(1)}R^\epsilon (\sum_{i=1}^N\|E_{S_i}F\|_{L^p(B_R)}^2)^{p/2}+K^{\frac{p}2-1}\sum_{\alpha\subset [0,1]^2}\int_{B_R}|E_{\alpha}F|^p.\end{equation}
The first term is the transverse contribution, estimated using Theorem \ref{t1} and the triangle inequality to rebuild the fractured pieces $S_i$ (no need for recoupling, as further losses in $K$ are harmless for this term).

For the second term, parabolic rescaling gives
$$\int_{B_R}|E_{\alpha}F|^p\lesssim C_{R/K^2}^p(\frac{N}{K})^{\frac{p}{2}-1}\sum_{S_i\cap \alpha\not=\emptyset}\int_{B_R}|E_{S_i\cap\alpha}F|^p.$$
Here $S_i\cap \alpha$ is the union of the $1/\sqrt{R}$-squares lying inside $\alpha$.
Under rescaling $\alpha$ becomes $[0,1]^2$, and the sets $S_i$ that intersect $\alpha$ become unions of  $K/\sqrt{R}$-squares. There are $O(N/K)$ such sets, and they satisfy the induction hypothesis relative to squares $Q$. The ball $B_R$ becomes a tubular region, covered by a finitely overlapping union of balls of radius $R/K^2$.

Thus, the contribution from the second term in \eqref{e1} is
$$\lesssim C_{R/K^2}^pN^{\frac{p}2-1}\sum_{S_i}\int_{B_R}\sum_{\alpha}|E_{\alpha\cap S_i}F|^p,$$
which is dominated via recoupling by
$$C_{R/K^2}^pN^{\frac{p}2-1}\sum_{S_i}\int_{B_R}|E_{S_i}F|^p.$$
We conclude that
$$C_R\le C_\epsilon K^{O(1)}R^\epsilon+C_1C_{R/K^2},$$
with $C_1$ independent of $K,R$.
Iterating this $R\to R/K^2\to R/K^4\ldots$ leads to $C_R\lesssim_\epsilon 1$.

\end{proof}
\medskip

Recoupling is typically lossy, more so than decoupling. The reason it is harmless in this argument is because its loss is compensated by the gain in the Bourgain--Guth argument, that comes from the one dimensional reduction to strips.

One example where Theorem \ref{t2} is applicable with $N\sim \sqrt{R}$ is when each $S_i$ consists of all the squares $q$ intersecting some curve $\gamma_i$. We assume each $Q$ with side length $L$ intersects at most $L\sqrt{R}$ such curves. If these curves are consecutive parallel lines at distance $\sim 1/\sqrt{R}$ from each other, the large ensembles $S_i$ are flat strips. But the curves can also be e.g. concentric circles.

We close this section with a few interesting questions  that the earlier argument seems insufficient to address.
\begin{qu}
Prove $l^2(L^3)$ decoupling into flat strips.
\end{qu}
Conjecture \ref{c2} may be more manageable for flat strips.
\begin{qu}
Prove $l^p(L^p)$ decoupling into flat strips for $3<p\le 10/3.$
\end{qu}

Let us  see why $p$ cannot be larger than $10/3$. Note that this does not follow from Corollary \ref{cosecon}. Indeed, the optimal range of $p$ for a given partition into $\sqrt{R}$-sets is sensitive to the geometry of the partition. For example, \eqref{sdcyre8we t} holds in the range $2\le p\le 4$ when each $P_j$ is an $R^{-1/4}$-square. We test the inequality for flat strips with $F=R1_{[0,1]^2}$. Each $E_qF$ is a single wave packet with magnitude $\sim 1$, so $|E_qF|\approx 1_{T_q}$, for some tube $T_q$ with radius $R^{1/2}$ and length $R$, containing the origin. By constructive interference we have
$|EF(x)|\sim R$ for $|x|\lesssim 1$,
so $$\|EF\|_{L^p(B_R)}\gtrsim R.$$
Using Cordoba's inequality for the parabola above the strip $S_i$, followed by the bush estimate for tubes, we find that for $p\le 4$
$$\|E_{S_i}F\|_{L^p(B_R)}\lesssim_\epsilon R^\epsilon\|(\sum_{q\subset S_i}|E_qF|^2)^{1/2}\|_{L^p(B_R)}\sim R^\epsilon\|\sum_{q\in S_i}1_{T_q}\|_{p/2}^{1/2}\sim R^{\epsilon} \|\sum_{q\in S_i}1_{T_q}\|_{1}^{1/p}\sim R^{\frac5{2p}+\epsilon}.$$
Apart from $R^\epsilon$, this computation is sharp.
Then \eqref{sdcyre8we t} forces that
$$R^p\lesssim_\epsilon R^\epsilon(\sqrt{R})^{p/2}R^{5/2},$$
or $p\le 10/3$.
\medskip

The sets covered by Theorem \ref{t2} are ``one-dimensional" in nature. The following spread-out case seems to lie at the opposite end of the spectrum.

Split $[0,1]^2$ into $R^{1/2}$ squares $Q$ with side length $R^{-1/4}$.   Each $Q$ is the union of $R^{1/2}$ canonical squares $q$. We choose each $S_i$ to contain   exactly one $q$ from each $Q$. This is connected to Theorem \ref{sparse}.
\begin{qu}
Prove $l^3(L^3)$ decoupling for these large ensembles.	
\end{qu}		
The following is worth noticing, as it shows that the contribution from each large ensemble is ``computable". As a consequence of Theorem \ref{candec} and Rubio de Francia's square function inequality, we have for $2\le p\le 4$ and each ball $B_{\sqrt{R}}$
$$
	\|E_{S_i}F\|_{L^p(B_{\sqrt{R}})}\approx \|(\sum_{q\in S_i}|E_qF|^{2})^{1/2}\|_{L^p(B_{\sqrt{R}})}.
$$

\section{Tight decoupling}

Let $p\ge 2$  and let $R$ be a large number. To motivate the forthcoming discussion, we ask how many points $\xi_1,\ldots,\xi_N\in \R^d$ with $|\xi_n|\lesssim 1$ can we find, so that we have $L^p$ square root cancellation on the ball $B(0,R)=\{x\in\R^d:\;|x|\le R\}$ (or equivalently, on the square $[0,R]^d$), namely
\begin{equation}
\label{sqrlp}
\frac{1}{|B(0,R)|}\int_{B(0,R)}|\sum_{n=1}^Na_n e(\xi_n\cdot x)|^pdx\lesssim \|a_n\|_{l^2}^p
\end{equation}
for each $a_n\in \C$. We ask that the implicit constant in \eqref{sqrlp} is independent of $R$. If \eqref{sqrlp} holds for some $p$, it will also hold for smaller indices, due to H\"older's inequality.

It is easy to see that the upper bound $N\lesssim R^{\frac{2d}{p}}$ is necessary, by considering  constant weights $a_n=1$. This is due to constructive interference near the origin
$$|\sum_{n=1}^Ne(\xi_n\cdot x)|\sim N,\;\;\text{for }|x|<1/100,$$
leading to the inequality
$$\frac{N^p}{R^d}\lesssim N^{p/2}.$$
Interestingly, there is a choice (in fact many) of $N\sim R^{2d/p}$ points $\xi_n$ such that \eqref{sqrlp} holds. This is a consequence of the $\Lambda(p)$ phenomenon, settled by Bourgain in \cite{Bo}.
\begin{te}
\label{tB}	
For each orthonormal system $\varphi_1,\ldots,\varphi_n$ of functions  with $\|\varphi_i\|_{L^\infty}\le 1$ and each $p>2$, there is a set $S\subset\{1,2,\ldots,n\}$ with size $\sim n^{2/p}$ such that
$$\|\sum_{i\in S}a_i\varphi_i\|_p\lesssim_p(\sum_{i\in S}|a_i|^2)^{1/2}$$
for each $a_i\in\C$. The implicit constant is independent of $n$ and $a_i$. A random subset $S$ with this size will work, with high probability. 	
\end{te}	
Consider the functions $\varphi_{\n}(x_1,\ldots,x_d)=e({n_1x_1+\ldots+n_dx_d})$ with $n_i$ positive integers smaller than $R$. They are $L^2$ orthogonal on $[0,1]^d$, and there are $\sim R^d$ of them.  Theorem \ref{tB} shows the existence of $\sim (R^d)^{2/p}$ such functions satisfying
$$\int_{[0,1]^d}|\sum_{\n\in S}a_\n e(\n\cdot x)|^pdx\lesssim \|a_\n\|_{l^2}^p.$$
Rescaling, it follows that taking $\xi_n$ to be the points in $\{\frac{\n}{R}:\;\n\in S\}$ satisfies \eqref{sqrlp}.
\smallskip

For generic $S$, the points $\frac{\n}{R}$ are spread out somewhat uniformly in the cube $[0,1]^d$. We will be interested in the case when the points lie on a prescribed manifold. Some degree of curvature will be needed, as the points cannot lie inside any given hyperplane $\H$. Indeed, if $\xi_n\in\H$ and $a_n=1$, we have constructive interference on the bigger set $E=\{x\in B(0,R):\;\dist(x,\H^\perp)\le 1/100\}$ of volume $\sim R$. This forces the more severe restriction $N\lesssim R^{\frac{2(d-1)}{p}}$.
\smallskip

If \eqref{sqrlp} is satisfied for a collection of $\xi_n$ with maximal size $N\sim R^{2d/p}$, we will refer to \eqref{sqrlp} as {\em tight decoupling.} Slightly weaker versions may also be considered, where $R^\epsilon$ losses are tolerated.

\medskip

Here is our main theorem in this section. To keep things simpler, we will restrict attention to the constant coefficient case, which is by far the most relevant for applications. But see also Question \ref{gen}.

\begin{te}[Tight decoupling for frequencies in thin sets]
	\label{ttight}	 Let $p\ge 2$ and let $E\subset \R^d$ be a compact set.
	Assume that there is a Borel probability  measure $d\sigma$ with $\supp(d\sigma)\subset E$ such that
\begin{equation}
\label{decay}
\int_{\R^d}|\widehat{d\sigma}(x)|^{p}dx<\infty.\end{equation}	
We also assume that for each $\epsilon$ there is a partition of $\supp(d\sigma)$ into Borel sets with equal $d\sigma$ measure and diameter less than $\epsilon$.

Then for each $R\gg 1$ there is a  collection of points $\Pc_R\subset E$ with cardinality $\sim R^{\frac{2d}{p}}$ such that we have square root cancellation on $B(0,R)$ in $L^p$
	 for  $a_\xi\equiv 1$
	\begin{equation}
\label{iucyf8rfodl;fk}
\frac{1}{|B(0,R)|}\int_{B(0,R)}|\sum_{\xi\in\Pc_R}a_\xi e(\xi\cdot x)|^pdx\lesssim \|a_\xi\|_{l^2}^p.\end{equation}

Conversely, if such $\Pc_R\subset E$ exists for a sequence of $R=R_k\to\infty$, then $E$ must support a probability measure $d\sigma$ satisfying $\eqref{decay}$.
\end{te}	
Inequality \eqref{decay} forces the support of $d\sigma$ to have Hausdorff dimension at least $2d/p$.

\begin{ex}[\textbf{Hypersurfaces}]
\end{ex}
Perhaps the most interesting class of examples comes from the case when $d\sigma$ is supported on a smooth hypersurface $\Sigma$. If we assume $\Sigma$ has nowhere zero Gaussian curvature, then it is known that the surface measure $d\sigma$ satisfies \eqref{decay} for each $p>\frac{2d}{d-1}$, since in fact
$|\widehat{d\sigma}(x)|\lesssim (1+|x|)^{-\frac{d-1}{2}}$. Theorem \ref{ttight} shows that tight decoupling is possible in this range.

On the other hand,  the Hausdorff dimension of $\Sigma$ is $d-1$, making it a Salem set. No measure supported on $\Sigma$ can satisfy \eqref {decay} with $p<\frac{2d}{d-1}$, so tight decoupling is impossible in this range.
\medskip

The case $p=\frac{4d}{d-1}$ is particularly revealing.
When the points $\xi_n\in\Sigma$ are $1/\sqrt{R}$-separated, it was proved in \cite{BD}, as a consequence of decoupling,  that \eqref{sqrlp} holds for $p\le \frac{2(d+1)}{d-1}$ for all $a_n$, albeit with mild $R^\epsilon$ losses. At this level of generality, this exponent is best possible.  To see this, assume $\Sigma$ is parametrized by $(x_1,\ldots,x_{d-1},\Phi(x_1,\ldots,x_{d-1}))$. Let $\xi_\n=(\frac{\n}{\sqrt{R}},\Phi(\frac{\n}{\sqrt{R}}))\in\Sigma$, with $\n\in \Z^{d-1}\cap B_{d-1}(0,\sqrt{R})$. These points are $1/\sqrt{R}$-separated, and there are $\sim R^{\frac{d-1}{2}}$ many of them. The new phenomenon is that the associated exponential sums have constructive interference in many regions of $B(0,R)$, not just near the origin. More precisely,
$$|\sum_{\n}e(\xi_\n\cdot x)|\sim R^{\frac{d-1}{2}}$$
whenever $|x_i-l_i\sqrt{R}|\le \frac1{100d}$ for each $1\le i\le d-1$ and some integer $|l_i|\lesssim \sqrt{R}$, and $|x_d|\le 1/100$. This set has measure $\sim R^{-\frac{d-1}{2}}$, which can be seen to force the restriction $p\le \frac{2(d+1)}{d-1}$ in \eqref{sqrlp}.

As a consequence of Theorem \ref{ttight} with $p=\frac{4d}{d-1}$, we find that each $\Sigma$ as above supports a collection $\Pc_R$ of roughly $R^{\frac{d-1}{2}}$ points $\xi$, such that we have square root cancellation in $L^{\frac{4d}{d-1}}$
$$
\frac{1}{|B(0,R)|}\int_{B(0,R)}|\sum_{\xi\in\Pc_R} e(\xi\cdot x)|^{\frac{4d}{d-1}}dx\lesssim R^{d}.
$$	
There is no guarantee that these points are $1/\sqrt{R}$-separated, but our proof of Theorem \ref{ttight} will show that this can be arranged to be true for a significant fraction of all pairs of points.

The examples provided by our proof of Theorem \ref{ttight} are non-deterministic, they arise via random selections. We are not aware of any deterministic example that is known to satisfy the case $p=\frac{4d}{d-1}$ of the theorem. However, we recall two conjectures that fit into this category. This time, we allow $R^\epsilon$ losses.

The first one is about lattice points on the sphere. Consider the typical eigenfunction of the Laplacian on $\T^d$, given by
$$S_{N,d}(\x)=\sum_{\n\in\sqrt{N}\S^{d-1}\cap \Z^d}a_{\n}e(\n\cdot \x).$$
It has been conjectured, see e.g. \cite{Bo112}, that if $d\ge 3$ and $p\le \frac{2d}{d-2}$
\begin{equation}
\label{consfere}
\int_{[0,1]^d}|S_{N,d}(x)|^{p}dx\lesssim_\epsilon N^\epsilon \|a_\n\|_{l^2}^p.\end{equation}
This inequality, sometimes referred to as the discrete Restriction conjecture for the sphere,  has  been verified in \cite{BD} in the smaller  range $p\le \frac{2(d+1)}{d-1}$, that is accessible via decoupling.

Let $d=3$, so  $\frac{2d}{d-2}=6$ coincides with $\frac{4d}{d-1}$ . Consider an integer $N$ for which $\sqrt{N}\S^{2}\cap \Z^3$ contains $\approx \sqrt{N}$ points (this is the case for infinitely many $N$). The rescaled collection $\frac1{\sqrt{N}}(\sqrt{N}\S^{2}\cap \Z^3)$ is conjectured to satisfy the tight decoupling \eqref{iucyf8rfodl;fk} with $R=\sqrt{N}$, $p=6$ and $R^\epsilon$ losses.

As a second example, we recall a conjecture from \cite{DG}. Consider the set of lattice points in the thin annulus
$$\Ac_{R,R^{-1/2}}=\{\n\in \Z^2:|\n-R|\le R^{-1/2}\}.$$
Moreover, we may assume that there are $\approx R^{1/2}$ such points. The rescaled points $\frac1{\sqrt{R}}\Ac_{R,R^{-1/2}}$ do not belong to the circle $\S^1$, but rather to a thin neighborhood of it. It is conjectured in \cite{DG} (Conjecture B) that
$$\int_{[0,1]^2} |\sum_{\n\in \Z^2:|\n-R|\le R^{-1/2}}e(\n\cdot x)|^8dx\lesssim_\epsilon R^{2+\epsilon}.$$
Thus, the rescaled points are expected to satisfy the tight decoupling \eqref{iucyf8rfodl;fk}
with $p=8$ (and $R^\epsilon$ losses). The arbitrary coefficient version of this inequality is false, as $\Ac_{R,R^{-1/2}}$ may contain many points in arithmetic progression.

When $d\ge 4$, \eqref{consfere} provides further candidates for tight decoupling, but the value of $\frac{2d}{d-2}$ is now smaller than $\frac{4d}{d-1}$. For infinitely many $N$, the set $\sqrt{N}\S^{d-1}\cap \Z^d$ contains $\approx N^{\frac{d-2}{2}}$ points. This matches $R^{2d/p}$ with $R=\sqrt{N}$ and $p=\frac{2d}{d-2}$. Interestingly, \eqref{consfere} was verified for $d=4$ and $a_\n=1$ in \cite{BD1}, using the Siegel mass formula. We are not aware of any other deterministic example for which tight decoupling is known to hold for frequencies localized on a hypersurface. This example gives the heuristics for why  \eqref{consfere} is extremely difficult. Tight decoupling is the highest degree on square root cancellation, and verifying it is expected to involve high-level number theory.

\begin{re}
In \cite{Ry} (see also Lemma 6 in \cite{LW} and the main result in \cite{K}), a probability measure $\nu$ supported on $\P^1$ is shown to exist for each $0<\alpha<1$ such that for each $\xi\in \supp(\nu)$ and $0<r\lesssim 1$
$$ r^{\alpha+\epsilon}\lesssim_\epsilon\nu(B(\xi,r))\lesssim_\epsilon r^{\alpha-\epsilon}$$
and
\begin{equation}
\label{kjrifurfigto[pgothob}
|\widehat{\nu}(x)|\lesssim_{\epsilon}(1+|x|)^{-\alpha/2+\epsilon}. \end{equation}
We use this with $\alpha=1/2$. Consider $R^{\frac12-\epsilon}\lesssim_\epsilon N_R\lesssim_\epsilon R^{\frac12+\epsilon}$ pairwise disjoint arcs $J$ on $\P^1$, with length $1/R$ and such that
$$R^{-\frac12-\epsilon}\lesssim_\epsilon \nu(J)\lesssim_\epsilon R^{-\frac12+\epsilon}.$$
Let $\psi_R(x)=e^{-(\frac{|x|}R)^2}$, so $\widehat{\psi_R}(\xi)=R^2e^{-(R{|\xi|})^2}$. Let $c_J$ be the center of $J$. Let $\eta$ be smooth, compactly supported, positive,  with $|\widehat{\eta}(x)|\ge 1$ for $|x|\le 1$. Note that for each $\xi\in\R^2$
$$0\le \sum_{J}\eta(R(\xi-c_J))\lesssim_\epsilon R^{-3/2+\epsilon}(d\nu)*\widehat{\psi_R}(\xi).$$
Since $8$ is an even integer, this implies that, taking Fourier transforms on both sides
$$\|{\frac1{R^2}\sum_{J}\widehat{\eta}(\frac{x}R)e(c_J\cdot x)}\|_{L^8(\R^2)}\lesssim_\epsilon R^{-3/2+\epsilon}\|\widehat{\nu}\psi_R\|_{L^8(\R^2)}\lesssim_\epsilon R^{-\frac32+\epsilon}.$$
The last inequality follows from \eqref{kjrifurfigto[pgothob}. We find that
$$\|\sum_{J}e(c_J\cdot x)\|_{L^8(B_R)}\lesssim_\epsilon R^{\frac12+\epsilon}.$$
This is tight decoupling with $R^\epsilon$ loss. The fact that $8$ is even integer is crucial to this argument.

This construction is also non deterministic, as $\nu$ is a random measure. The frequency points $c_J$ belong to a fractal. This is very different from our main construction, where $\xi_n$ are well spaced, with high probability.
\end{re}
We close the discussion about hypersurfaces with the following question, open even for $\P_1$.
\begin{qu}
\label{gen}	
Let $\frac{2(d+1)}{d-1}\le p\le \frac{4d}{d-1}$. Given a hypersurface $\Sigma\subset \R^d$ with positive curvatures (e.g. the sphere), and $R\gg 1$, is there a set $\Pc_R\subset \Sigma$ with $\sim R^{2p/d}$ points such that $$\frac{1}{|B(0,R)|}\int_{B(0,R)}|\sum_{\xi\in\Pc_R}a_\xi e(\xi\cdot x)|^pdx\lesssim \|a_\xi\|_{l^2}^p$$
holds for arbitrary $a_\xi$? How about with $R^\epsilon$ losses, when $p>\frac{2(d+1)}{d-1}$?	
\end{qu}
The additional restriction $p\ge \frac{2(d+1)}{d-1}$ is forced
by lower dimensional concentration. There is a cap with diameter $\sim R^{-1/2}$ on $\Sigma$ containing $N\sim R^{\frac{2d}{p}-\frac{d-1}{2}}$ of the points in $\Pc_R$. Test the inequality with $a_\xi=1$ on this set, and zero otherwise. The points lie inside a box with dimensions $R^{-1/2}$ ($d-1$ times) and $1/R$ (once). There is constructive interference on the box centered at the origin  with dimensions $R^{1/2}$ ($d-1$ times) and $R$ (once). Thus
$$R^{-d}R^{\frac{d-1}2+1}N^p\lesssim N^{p/2}, $$
which means $p\ge \frac{2(d+1)}{d-1}$.

\begin{ex}[\textbf{Curves}]
\end{ex}	
	Consider a curve with nonzero torsion in $\R^d$, such as the moment curve
$$\Gamma^d=\{(t,\ldots,t^d):\;0\le t\le 1\}.$$
The main result in \cite{GIZZ} proves that $\Gamma^d$ does not support any measure with Fourier transform in $L^p$ for   $p<\frac{d^2+d+2}{2}$. The second part of Theorem \ref{ttight} concludes that it is impossible to achieve tight decoupling in the range $p<\frac{d^2+d+2}{2}$, with frequencies in $\Gamma^d$.

On the other hand, it is known (\cite{ACK}) that the uniform measure on $\Gamma^d$  satisfies
\eqref{decay} for $p>\frac{d^2+d+2}{2}$. When combined with the first part of Theorem \ref{ttight}, we see that tight decoupling is possible in this range.

Let us specialize to  $p=2d^2$. Theorem \eqref{ttight} shows the existence of roughly $R^{1/d}$ points on $\Gamma^d$ such that \eqref{iucyf8rfodl;fk} holds for this $p$. On the other hand it was proved in \cite{BDG} that \eqref{iucyf8rfodl;fk} holds for $p\le d(d+1)$, arbitrary coefficients and arbitrary collections $\Pc_R$ of points on $\Gamma^d$ that are $1/R^{1/d}$-separated. This range follows from decoupling and is sharp under just the separation condition, as it can be seen by considering the rescaled lattice points on $\Gamma^d$ corresponding to $t=n/R^{1/d}$, $1\le n\le R^{1/d}$.
\medskip

We use two classical results. The first one is due to Marcinkiewicz--Zygmund.
\begin{te}
\label{tp1}	
Let $X_1,\ldots,X_N$ be independent, mean zero  random variables belonging to $L^p$ for some $p\ge 1$. Then
$$\int|\sum_{n=1}^NX_n|^{p}\sim_p \int(\sum_{n=1}^N|X_n|^2)^{p/2}.$$
\end{te}
The second one is usually attributed to Rosenthal, see Lemma 1 in \cite{Ro}. It does not require mean zero.
\begin{te}
\label{tp2}		
Let $Y_1,\ldots,Y_N$ be independent,  random variables belonging to $L^q$ for some $q\ge 1$. Then
$$\int|\sum_{n=1}^NY_n|^{q}\lesssim_q\sum_{n=1}^N\int|Y_n|^q+(\sum_{n=1}^N\int|Y_n|)^{q}.$$
\end{te}
Let us now prove the first part in Theorem \ref{ttight}.

Consider a partition of $\supp(d\sigma)$ into sets $S_1,\ldots,S_N$ with $d\sigma$ measure equal to $1/N$ and diameter $\lesssim R^{-1-\frac{d}{p}}$. Pick any point $\xi_n\in S_n$. Then, if $|x|<R$
\begin{align*}
|\frac1{N}\sum_{n=1}^Ne(\xi_n\cdot x)-\widehat{d\sigma}(x)|&\le \sum_{n=1}^N|\int_{S_n}(e(\xi_n\cdot x)-e(\xi\cdot x))d\sigma(\xi)|\\&\le \frac1N\sum_{n=1}^N\max_{\xi\in S_n}|e(\xi_n\cdot x)-e(\xi\cdot x)|
\\&\lesssim \max_n\max_{\xi\in S_n}R|\xi_n-\xi|\lesssim R^{-d/p}.
\end{align*}
Combining this with \eqref{decay} it follows that
$$\int_{B(0,R)}|\sum_{n=1}^Ne(\xi_n\cdot x)|^pdx\lesssim  N^p(1+\int_{B(0,R)}|\widehat{d\sigma}(x)|^pdx)\lesssim N^p.$$
Assume $\frac1N\le \delta\le 1$.
Consider the $\{0,1\}$-valued independent random variables $f_n(\omega)$ with mean $\delta$, defined on the probability space $\Omega$. We first write for each $x\in B(0,R)$
$$\int_{\Omega}|\sum_{n=1}^Nf_n(\omega)e(\xi_n\cdot x)|^pd\omega\lesssim\delta^p|\sum_{n=1}^Ne(\xi_n\cdot x)|^p+\int_{\Omega}|\sum_{n=1}^N(f_n(\omega)-\delta)e(\xi_n\cdot x)|^pd\omega.$$
For the second term, we use Theorem \ref{tp1} with $X_n(\omega)=(f_n(\omega)-\delta)e(\xi_n\cdot x)$ and then Theorem \ref{tp2} with $q=p/2$ and $Y_n(\omega)=|f_n(\omega)-\delta|^2$
\begin{align*}
\int_{\Omega}|\sum_{n=1}^N(f_n(\omega)-\delta)e(\xi_n\cdot x)|^pd\omega&\lesssim \int_\Omega(\sum_{n=1}^N|f_n(\omega)-\delta|^2)^{p/2}d\omega\\&\lesssim \sum_{n=1}^N\int_\Omega|f_n(\omega)-\delta|^pd\omega+(\sum_{n=1}^N\int_\Omega|f_n(\omega)-\delta|^2d\omega)^{p/2}\\&\lesssim
N(\delta+\delta^p)+(N(\delta+\delta^2))^{p/2}\sim (N\delta)^{p/2}.
\end{align*}

Integrating on $B(0,R)$ we get
$$\frac1{|B(0,R)|}\int_\Omega\int_{B(0,R)}|\sum_{n=1}^Nf_n(\omega)e(\xi_n\cdot x)|^pdxd\omega\lesssim (N\delta)^{p/2}+\frac{(N\delta)^p}{R^d}.$$
We select $\delta$ such that $N\delta\sim R^{2d/p}$. Let
$$\Omega'=\{\omega\in\Omega:\;|\sum_{n=1}^Nf_n(\omega)-N\delta|\ge \frac{N\delta}{2}\}.$$
Then
$$|\Omega'|\le  \frac{4\int_\Omega|\sum_{n=1}^N(f_n(\omega)-\delta)|^2d\omega}{N^2\delta^2}=\frac{4N(\delta+\delta^2)}{N^2\delta^2}\lesssim \frac1{N\delta}.$$
Thus $|\Omega\setminus \Omega'|\ge 1/2$ if $R$ is large enough. We may thus pick $\omega\in\Omega\setminus\Omega'$ such that
$$\frac1{|B(0,R)|}\int_{B(0,R)}|\sum_{n=1}^Nf_n(\omega)e(\xi_n\cdot x)|^pdx\lesssim (N\delta)^{p/2}.$$
It suffices to note that $f_n(\omega)=1$ for $\sim N\delta$ values of $n$. The set $\Pc_R$ will consist precisely of those $\xi_n$ corresponding to these $n$.
\medskip

We close with the proof of the second part in Theorem \ref{ttight}. For each $R=R_k$, consider the probability measure on $E$
$$\mu_k=\frac1{|\Pc_{R_k}|}\sum_{\xi\in\Pc_{R_k}}\delta_{\xi}.$$
Note that
$$\widehat{\mu_k}(0)=1$$
and
$$\int_{B(0,R_k)}|\widehat{\mu_k}(x)|^pdx\lesssim 1.$$
Consider a weak limit point $d\sigma$ of the sequence. Since $\widehat{\mu_k}(x)\to \widehat{d\sigma}(x)$ for each $x$, we find that  $\sigma$ is a probability measure. Also, the dominated convergence theorem shows that for each ball $B(0,r)$
$$\int_{B(0,r)}|\widehat{d\sigma}(x)|^pdx=\lim_{k\to\infty}\int_{B(0,r)}|\widehat{\mu_k}(x)|^pdx\le\lim_{k\to\infty}\int_{B(0,R_k)}|\widehat{\mu_k}(x)|^pdx\lesssim 1.$$

\end{document}